\documentclass[11pt]{amsart}
\usepackage[english]{babel}
\usepackage{graphics}
\usepackage{epsfig}
\usepackage{latexsym,amssymb,amsfonts,amsmath, amscd, euscript,  MnSymbol}
\usepackage[all]{xy}

\usepackage{enumerate}

\newcommand{\N}{\mathbb{N}}

\newcommand{\Z}{\mathbb{Z}}
\newcommand{\Q}{\mathbb{Q}}

\newtheorem{definition}{{\bf Definition}}

\newtheorem{theorem}{{\bf Theorem}}
\newtheorem{proposition}{\noindent {\bf Proposition}}
\newtheorem{corollary}{\noindent {\bf Corollary}}

\newtheorem{claim}{\noindent {\bf Claim}}

\newtheorem{example}{\noindent {\bf Example}}
\newtheorem{problem}{\noindent{\bf Problem}}

\def\proofref #1 {{\noindent  {\bf Proof} (#1).}\ }

\newtheorem{lemma}[definition]{\noindent {\bf Lemma}}
\def\endproof{\hfill {\kern 6pt\penalty 500
\raise -0pt\hbox{\vrule \vbox to5pt {\hrule width 5pt
\vfill\hrule}\vrule}}}
\def\centerpicture #1 by #2 (#3){\leavevmode
        \vbox to #2{
        \hrule width #1 height 0pt depth 0pt
        \vfill
        \special{pictfile #3}}}

\begin{document}

\title{Isomorphy up to complementation}

\author{Maurice Pouzet}\address{ICJ, Math\'ematiques, Universit\'e Claude-Bernard Lyon1,
43 Bd 11 Novembre 1918, 69622 Villeurbanne Cedex, France,  and Mathematics \& Statistics Department, University of Calgary, Calgary, Alberta, Canada T2N 1N4 }
\email{pouzet@univ-lyon1.fr}
\author{Hamza Si Kaddour}
\address{ICJ, Math\'ematiques, Universit\'e Claude-Bernard Lyon1,
43 Bd 11 Novembre 1918, 69622 Villeurbanne Cedex, France}
\email{sikaddour@univ-lyon1.fr}



\date{\today}

\dedicatory{Dedicated with warmth and admiration to Adrian Bondy at the occasion of his seventieth birthday}.

\begin{abstract}   Considering uniform hypergraphs, we prove  that for every non-negative integer $h$ there exist two non-negative integers  $k$ and $t$ with  $k\leq t$ such that two $h$-uniform hypergraphs ${\mathcal H}$ and ${\mathcal H}'$ on the  same set $V$ of vertices, with $\vert V\vert \geq t$, are equal up to complementation whenever ${\mathcal H}$ and ${\mathcal H}'$ are $k$-{hypomorphic up to complementation}. Let  $s(h)$ be the  least integer $k$ such that the conclusion above holds and let $v(h)$ be the least $t$ corresponding to $k=s(h)$. We prove that $s(h)= h+2^{\lfloor \log_2 h\rfloor}
$. In the special case $h=2^{\ell}$ or $h=2^{\ell}+1$, we prove that  $v(h)\leq s(h)+h$.  The values $s(2)=4$ and $v(2)=6$ were obtained in  \cite{dlps1}. 
 \end{abstract}

 \maketitle

\section{Main results}

We extend to hypergraphs a reconstruction result about graphs obtained in \cite{dlps1}. 

We start recalling the result. Let $V$ be a set of cardinality $v$ (possibly infinite). Two graphs $G$ and $G'$ with vertex set $V$ are  {\it isomorphic up to complementation} if $G'$ is isomorphic to $G$ or to the complement 
$\overline G$ of $G$. Let $k$ be a non-negative integer, $G$ and $G'$ are {\it $k$-hypomorphic  up to complementation} if for every $k$-element subset $K$ of $V$, the induced subgraphs $G_{\restriction K}$ and $G'_{\restriction K}$  are isomorphic up to complementation.  A graph $G$ is {\it $k$-reconstructible up to complementation} if every  graph  $G'$ which is  $k$-hypomorphic to   $G$ up to complementation is in fact isomorphic to $G$ up to complementation. 

It is shown in \cite{dlps1} that   two graphs $G$ and $G'$ on the same set of $n$ vertices are equal up to complementation whenever they are $k$-hypomorphic up to complementation and  $4\leq k\leq n-4$ (see \cite {dlps2} for the case $k=n-3$). It is also shown  that  $4$ is  the least integer $k$ such that every graph $G$ having a large number $n$ of vertices is
$k$-reconstructible up to complementation. 

We show here that this result extends to uniform hypergraphs.  Our result is the following (see the definitions in the next section). 

\begin{theorem}\label{thm1} Let $h$ be a non-negative integer. There are two non-negative integers $k$ and $t$, $k\leq t$ such that two $h$-uniform hypergraphs ${\mathcal H}$ and ${\mathcal H}'$ on the  same set $V$ of vertices, $\vert V\vert \geq t$, are equal up to complementation whenever ${\mathcal H}$ and ${\mathcal H}'$ are $k$-{hypomorphic up to complementation}.
\end{theorem}
Let $s(h)$ be the  least integer $k$ such that there is some $t$ such that the the conclusion above holds and let $v(h)$ be the least $t$ corresponding to $k=s(h)$.
 
\begin{theorem} \label{thm-main2}
\begin{enumerate}
\item $s(h) = h+2^{\lfloor \log_2 h\rfloor}$; 

\item $v(h)\leq s(h)+h$ if $h= 2^{\ell}$ or $h= 2^{\ell}+1$. 
\end{enumerate}			
\end{theorem}

\begin{problem}
Does $v(h)=s(h)+h$ for every $h$?
\end{problem}

The proof of the result of \cite{dlps1} was based on a result of Wilson on the rank of incidence matrices over the two-element field \cite{W}. Here,  we use  essentially Ramsey's theorem, Lucas's theorem and the notion of almost constant hypergraph. We use Ramsey's theorem and compactness theorem of first order logic in order to obtain the existence of $k$ and $t$ (see Theorem \ref{FRA} and Claim \ref{claim:compactness}). We use Lucas's theorem to prove $(1)$ of Theorem  \ref{thm-main2} (we use it for the inequality 
$s(h)\geq h+2^{\lfloor \log_2 h\rfloor}$ (see Theorem \ref{thm2}) and also for the reverse inequality (see Theorem \ref{seuil}), where we use also  a decomposition result about hypergraphs (see Proposition \ref{maurice-thiery})).  With Wilson's theorem, we get  both $(1)$ and $(2)$ when $h$ is of the form $2^{\ell}$ or $2^{\ell}+1$
(see Corollary \ref{$k=2h$ et $v=3h$} and Theorem \ref{thm h=3, k=4;5}). 

In the last section of the paper we introduce a generalization of the notion of isomorphy up to complementation. We consider  colorations of  complete  graphs and hypergraphs and isomorphy up to a permutations of the colors. We raise some questions.

We   refer to \cite{Bo} for notions of graph theory. We use also basic notions of set theory, we denote by $\powerset (V)$ the power set of $V$ and by $[V]^h$ the collection of $h$-element subsets of $V$.

\section{Hypergraphs, incidence matrices and almost constant  hypergraphs}

\subsection{Isomorphy up to complementation}
 
Recall that a {\it hypergraph} is a pair ${\mathcal H}: =(V,{\mathcal E})$ where ${\mathcal E}$ is a collection of subsets of $V$; members of $V$ are the {\it vertices} of  ${\mathcal H}$, whereas members of ${\mathcal E}$ are the {\it hyperedges}. We denote by $V(\mathcal H)$, resp.   $\mathcal E(\mathcal H)$,  the sets of vertices, resp.  hyperedges,  of a  hypergraph $\mathcal H$; we denote by $v(\mathcal H)$, resp. $e(\mathcal H)$, the cardinality of $V(\mathcal H)$, resp. $\mathcal E(\mathcal H)$. 
If $E\in \mathcal E({\mathcal H})$ we set ${\mathcal H}(E)=1$; otherwise we set ${\mathcal H}(E)=0$. 
If $K$ is a subset of $V$, the 
{\it  hypergraph induced by} ${\mathcal H}$ {\it on} $V$ is ${\mathcal H}_{\restriction K}= (K, {\mathcal E}\cap {\powerset (K)})$. 
Let $h$ be an integer; the hypergraph  ${\mathcal H}$  is $h$-{\it uniform} (or \emph{$h$-regular}) if all its edges have size $h$ (for instance every graph is a $2$-uniform hypergraph); we make the convention that a hypergraph with no hyperedge is $h$-uniform for every $h$. With this convention, if  ${\mathcal H}$ is $h$-uniform, the {\it complement of} ${\mathcal H}$,  $\overline{\mathcal H}:=(V, [V]^h\setminus {\mathcal E})$ is $h$-uniform. 

 Let  ${\mathcal H}:=(V,{\mathcal E})$,  ${\mathcal {H}'}:=(V',{\mathcal E}')$, be two hypergraphs. An {\it isomorphism} from 
${\mathcal H}$ onto ${\mathcal H}'$ is any bijective map $f$ from $V$ onto $V'$ such that the natural extension $\overline f$ to $\powerset (V)$  induces a  bijective map from ${\mathcal E}$ onto 
${\mathcal E}'$.  If such a map exists,  ${\mathcal H}$ and ${\mathcal H}'$ are {\it isomorphic}. They  are {\it isomorphic up to complementation} if either ${\mathcal H}$ is isomorphic to  ${\mathcal H}'$ or  ${\mathcal H}$ is isomorphic to  $\overline{{\mathcal H}'}$.  If ${\mathcal H}$ and ${\mathcal H}'$ have the same set $V$ of vertices, we say that there are \emph{equal up to complementation} if ${\mathcal H}= {\mathcal H}'$ or ${\mathcal H}= \overline{{\mathcal H}'}$; if moreover, the induced hypergraphs   ${\mathcal H}_{\restriction K}$ and  ${{\mathcal H}'}_{\restriction K}$ are  isomorphic up to complementation  for all the $k$-element subsets of $V$, we say that  ${\mathcal H}$ and ${\mathcal H}'$ are $k$-{\it hypomorphic up to complementation}.

The relationship between $k$-hypomorphy for different values of $k$ is given by Proposition \ref {thm:down-k-iso}  below. The case of graphs is treated by Proposition 2.4 \cite {dlps1} by means of Gottlieb-Kantor theorem \cite{Go, KA}. The general case follows the same lines.

\subsection{Incidence matrices}
 
Let $V$ be a finite set, with $v$ elements. Given non-negative integers $t,k$, let $W_{t\;k}$ be the  ${v \choose t}$  by  ${v \choose k}$ matrix of $0$'s and $1$'s, the rows of which are indexed by the $t$-element subsets
 $T$ of $V$, the columns are indexed by the $k$-element subsets $K$
of $V$, and where the entry $W_{t\;k}(T,K)$ is $1$ if
 $T\subseteq K$ and is $0$ otherwise.

A fundamental result, due to D.H.Gottlieb \cite{Go}, and independently W.Kantor \cite {KA}, is this:

\begin{theorem}\label{gottlieb-kantor} For $t\leq min{(k,v-k)}$,  $W_{t\; k}$ has full row rank over the field $\Q$ of rational numbers.
 \end{theorem}

 If $k:=v-t$ then, up to a relabelling, $W_{t\; k}$ is the adjacency matrix $A_{t,v}$
 of the  {\it Kneser graph} $KG(t,v)$, graph  whose vertices are the $t$-element
subsets of $V$, two subsets forming an edge if they are disjoint.

An equivalent form of Theorem \ref{gottlieb-kantor} is:
\begin{theorem}  \label{Ka} $A_{t, v}$ is non-singular for $t\leq \frac{v}{2}$.
\end{theorem}

Applications to graphs and relational structures where given in  \cite{Fr2}  and \cite{Pm}. Let us explain why the use of this result in our context is natural.\\
Let $X_1,\ldots ,X_r$ be an
enumeration of the   $h$-element subsets of $V$; let $K_1,\ldots ,K_{s}$ be an enumeration of the   $k$-element
subsets of $V$ and  $W_{h\; k}$ be the matrix of the $h$-element subsets versus the $k$-element subsets.  If ${\mathcal H}$  is a $h$-uniform hypergraph with vertex set $V$, let $w_{\mathcal H}$ be the row matrix $(g_1,\ldots , g_r)$ where $g_i=1$ if  $X_i$ is a hyperedge of ${\mathcal H}$,  $0$ otherwise.   We have
   $w_{\mathcal H}W_{h\; k}=(e({\mathcal H}_{\restriction K_1}),\ldots , e({\mathcal H}_{\restriction K_s}))$. Thus, if ${\mathcal H}$ and ${\mathcal H}'$ are two hypergraphs with vertex set $V$ such that  ${\mathcal H}_{\restriction K}$ and ${\mathcal H}'_{\restriction K}$ have the same number of hyperedges  for every $k$-element subset of $V$, we have  $(w_{\mathcal H}-w_{{\mathcal H}'})W_{h\; k}=0$. Thus, provided that $v\geq k+h$, by Theorem \ref{gottlieb-kantor}, $w_{\mathcal H}-w_{{\mathcal H}'}=0$ that is ${\mathcal H}={\mathcal H}'$.
   
   \begin{proposition} \label {thm:down-k-iso}
Let $v,k$ be non-negative integers, Let  $t \leq min{(k,  v-k)}$ and ${\mathcal H}$ and ${\mathcal H}'$ be two  $h$-uniform hypergraphs on the same set $V$ of $v$ vertices.  If ${\mathcal H}$ and ${\mathcal H}'$ are $k$-hypomorphic up to complementation then they are $t$-hypomorphic up to complementation.
\end{proposition}
\begin{proof} Let ${\mathcal G}$ be a hypergraph on $t$ vertices. Set $Is ({\mathcal G}, {\mathcal H}):= \{L\subseteq V: {\mathcal H}_{\restriction L}\simeq {\mathcal G}\}$,  $Isc({\mathcal H},{\mathcal G}):=Is({\mathcal G},H)\cup Is(\overline {\mathcal G},{\mathcal H})$ and $w_{{\mathcal G},{\mathcal H}}$ the $0-1$-row vector indexed by the $t$-element subsets $X_1,\ldots, X_r$ of $V$ whose coefficient of $X_i$ is $1$ if $X_i\in Isc({\mathcal G},{\mathcal H})$ and $0$ otherwise.  From our  hypothesis, it follows that  $w_{{\mathcal G},{\mathcal H}} W_{t \; k}= w_{{\mathcal G},{\mathcal H}'} W_{t \; k}$. From Theorem \ref{gottlieb-kantor}, this implies $w_{{\mathcal G},{\mathcal H}}= w_{{\mathcal G},{\mathcal H}'}$ that is $Isc({\mathcal G},{\mathcal H})=Isc({\mathcal G},H')$. Since this equality holds  for all hypergraphs ${\mathcal G}$ on $t$-vertices, the conclusion of the proposition follows.
\end{proof}\\

In particular, two $h$-uniform hypergraphs ${\mathcal H}$ and ${\mathcal H}'$ on the  same set $V$ of vertices, $\vert V\vert \geq 2k-1$, are $k'$-hypomorphic  up to complementation for every $k'\leq k$ provided that there are  $k$-hypomorphic up to complementation.\\

A fundamental result, due to R.M.Wilson \cite{W}, is the following.

\begin{theorem} (R.M. Wilson \cite{W}) \label{thm Wilson} For $t\leq min{(k,v-k)}$, the rank of $W_{t\; k}$ modulo a prime $p$ is
$$ \sum  {v  \choose i}-  {v\choose
i - 1}
$$
where the sum is extended over those indices $i$, $0\leq i\leq t$,  such that $p$
does not divide the binomial coefficient ${k-i \choose
t-i}$.
\end{theorem}

In the statement of the theorem,  ${v \choose -1}$
should be interpreted as zero.\\

We will apply Wilson's theorem to $h$-uniform hypergraphs, for $p=2$, $t=h$, $k=2rh$ where $r$ is a positive integer.
In this case we will obtain that the rank of  $W_{h\; k}$ $(mod \ 2)$ is
${v  \choose h} -1$.

Let $n,p$ be positive integers, the decomposition of $n=\sum_{i=0}^{n(p)} n_i p^i$ in the basis $p$ is also denoted $[n_0,n_1,\dots ,n_{n(p)}]_p$  where $n_{n(p)}\neq 0$ if and only if $n\neq 0$.

\begin{theorem} (Lucas's Theorem  \cite{Lucas})    \label{lucas}
Let $p$ be a prime number, $t,k$ be   positive integers,
$t\leq k$, $t=[t_0,t_1,\dots ,t_{t(p)}]_p$ and $k=[k_0,k_1,\dots ,k_{k(p)}]_p$. Then
$${k  \choose t} = \prod_{i=0}^{t(p)} {k_i  \choose t_i} \  (mod \ p),\ \mbox{where}  \
{k_i  \choose t_i} =0\ \mbox{if} \ t_i>k_i.$$
\end{theorem}

For an elementary proof of  Theorem \ref{lucas}, see Fine \cite{Fine}.
We will use the following well-known consequences  of Theorem \ref{lucas}.

\begin{corollary}\label{cor-lucas}
\begin{enumerate}
\item Let $p$ be a prime and $t,k$ be   positive integers, 
$t\leq k$, let $t=[t_0,t_1,\dots ,t_{t(p)}]_p$ and $k=[k_0,k_1,\dots ,k_{k(p)}]_p$. 
\begin{enumerate}
 \item Then $p \vert {k  \choose t}$ if and only if there is $i\in \{0,1,\dots , t(p)\}$ such that $t_i>k_i$.
\item ${k \choose t}$  is odd if, and only if, for all integer $i\in \{0,1,\dots , t(2)\}$, $t_i = 1 \Rightarrow  k_i = 1$.
\end{enumerate}
\item Let $v\geq 2$. Then $v$ is a power of $2$
if, and only if, ${v\choose k}$  is even for all $k\in \{1,\ldots ,v-1\}$.
\end{enumerate}
\end{corollary}

\begin{corollary}\label{rank} Let $h$ and $k$ be positive integers.
 Then  
the rank of $W_{h\; k}$ modulo $2$ is  ${v\choose h}-1$ if and only if  $h$ is a power of $2$ and $k=2rh$ where $r$ is a positive integer.
\end{corollary}

\begin{proof} First we prove the converse implication. We have $h=2^{\ell}$ for some integer $\ell$ and $r=\sum_{j=0}^{t} \varepsilon_j 2^{j}$ for some $t$ with $\varepsilon_t=1$, then $k=\sum_{j=0}^{t} \varepsilon_j 2^{\ell +j+1}$. Let $s$ be the first integer $j$ such that   $\varepsilon_j\neq 0$. 
We have $2^{\ell +s+1}=2^{l}+\sum_{p=0}^{s} 2^{l+p}$.
For $i\leq h$, ${{k-i}  \choose {h-i}}={{(2^{l}-i)+\sum_{p=0}^{s} 2^{l+p}+ \sum_{j=s+1}^{t} \varepsilon_j 2^{\ell +j+1}}  \choose {2^{l}-i}}$. 
Applying Lucas's theorem, ${{k}  \choose {h}}$ is even and ${{k-i}  \choose {h-i}}$ is odd if $i\neq 0$. 
Now by Wilson's theorem,  
the rank of $W_{h\; k}$ modulo $2$ is  ${v\choose h}-1$.  \\
Let us prove the direct implication. 
Note that the rank of $W_{h\; k}$ modulo $2$ is  ${v\choose h}-1$ if and only if  ${k-i \choose
h-i}$ is odd for all $i\in \{1,\ldots ,h\}$ and ${k \choose
h}$ is even, in particular $k>h$. 
We have $h{k\choose h}=k{{k-1}\choose {h-1}}$. Since ${k\choose h}$ is even and  ${{k-1}\choose {h-1}}$ is odd then $k$ is even. 
Now  $(h-1){{k-1}\choose {h-1}}= (k-1){{k-2}\choose {h-2}}$. 
Since ${{k-1}\choose {h-1}}$ and ${{k-2}\choose {h-2}}$ are odd then $h$ and $k$ have the same parity. So $h$ and $k$ are even.\\
We have $h=[h_0,h_1,\dots ,h_{h(2)}]_2$ and $k=[k_0,k_1,\dots ,k_{k(2)}]_2$ with $h_0=k_0=0$.\\
First we prove that for all $i< h_{h(2)}$, $h_i=k_i$. By contradiction, let $j$ be the first integer $i\geq 1$ such that $h_i\neq k_i$. If $h_j=0$ and $k_j=1$, then $(h-2^{j})_j=1$ and $(k-2^{j})_j=0$, this contradicts the fact that ${{k-2^{j}}\choose {h-2^{j}}}$ is odd. So $h_j=1$ and $k_j=0$. Since ${{k-2^{h(2)}}\choose {h-{2^{h(2)}}}}$ is odd then  $k_{h(2)}=0$. 
Let $n$ be the first integer $m> h(2)$ such that $k_m\neq 0$. We have 
$2^{n}-2^{h(2)} =2^{h(2)} + 2^{h(2)+1}+\ldots +2^{n-1}$. 
Then $(k-2^{h(2)})_{h(2)}=1$ and $(h-2^{h(2)})_{h(2)}=0$. That contradicts  the fact that
	${k-2^{{h(2)}}\choose {h-2^{{h(2)}}}}$ is odd. 
So we have proved that for all $i< h_{h(2)}$, $h_i=k_i$. Then, since ${k\choose h}$ is even, we have  $k_{h(2)}=0$. Thus, for all $i< h_{h(2)}$, $h_i=k_i=0$ since 
	  ${{k-2^i}\choose {h-2^i}}$ is odd. 
	 That gives $h=2^{\ell}$ for some integer $\ell$ and $k=\sum_{i=\ell +1}^{t} k_i2^i$ for some integer $t$. Then $k=2rh$ where $r$ is a positive integer.
\end{proof}

\begin{proposition}\label{$k=2h$ et $v=3h$ parite} Let $h$ be a power of $2$, $k=2rh$ where $r$ is a positive integer, and ${\mathcal H}$ and ${\mathcal H}'$ be two $h$-uniform hypergraphs   on the same set $V$ of $v\geq k+h$ vertices.  
Then the following properties are equivalent:\\
(i) $e({\mathcal H}_{\restriction K})$ and   $e({\mathcal H}'_{\restriction K})$ have the same parity  for  all $k$-element subsets $K$ of $V$;\\
(ii) ${\mathcal H}'= {\mathcal H}$ or ${\mathcal H}'= \overline {\mathcal H}$.
\end{proposition}

\begin{proof}
The implication $(ii)\Rightarrow (i)$ is trivial. We prove  $(i)\Rightarrow (ii)$.\\
We have $h=2^l$ for some integer $l$. 
Let $W_{h\; k}$ be the matrix defined page 3 and $^tW_{h\; k}$ its transpose.
Let $U:= {\mathcal H}\dot{+} {\mathcal H}'$. From the fact that $e({\mathcal H}_{\restriction K})$ and $e({\mathcal H}'_{\restriction K})$ have the same parity for  all $k$-element subsets $K$, the boolean sum $U$ belongs to the kernel of  $^tW_{h\; k}$ over the $2$-element field. 
By Corollary \ref{rank},  the rank of $W_{h\; k}$ modulo $2$ is  ${v\choose h}-1$.  
Then the kernel of  its transpose $^tW_{h\; k}$
 has dimension $1$. 
 Since $(1,\cdots ,1)W_{h\; k}=(0,\cdots ,0)$ $(mod \ 2)$
  then   $w_UW_{h\; k}=(0,\cdots ,0)$ $(mod \ 2)$ amounts to $w_U=(0,\cdots ,0)$ or  $w_U=(1,\cdots ,1)$, that is  $U$ is empty or complete, so  ${\mathcal H}'={\mathcal H}$ or ${\mathcal H}'=\overline {\mathcal H}$.
 \end{proof} 
  
  \begin{corollary} \label{$k=2h$ et $v=3h$} 
Let $h$ be a power of $2$, $k=2rh$ where $r$ is a positive integer, and   
 ${\mathcal H}$ and ${\mathcal H}'$ be two $h$-uniform hypergraphs on the same set $V$ of $v\geq k+h$ vertices. If  ${\mathcal H}$ and ${\mathcal H}'$ are $k$-hypomorphic up to complementation then ${\mathcal H}'= {\mathcal H}$ or ${\mathcal H}'= \overline {\mathcal H}$.
\end{corollary}
 
 \begin{proof} From Corollary \ref{cor-lucas}, ${{k}\choose h}$ is even. Then 
we conclude using Proposition \ref{$k=2h$ et $v=3h$ parite}.
 \end{proof}

\begin{theorem}\label{thm h=3, k=4;5} 
Two $(2^{\ell}+1)$-uniform hypergraphs ${\mathcal H}$ and ${\mathcal H}'$ on the  same set $V$ of $v$ vertices,  $v \geq 3. 2^{\ell}+2$ and $\ell\geq 1$, are equal up to complementation whenever ${\mathcal H}_{\restriction K}$ and ${\mathcal H}'_{\restriction K}$ have the same number of edges up to complementation for all $k$-element subsets $K$ of $V$ for $k\in \{2^{\ell +1},2^{\ell +1}+1\}$.
\end{theorem}

\begin{proof} 
Let $U:= {\mathcal H}\dot{+} {\mathcal H}'$. By Corollary \ref{cor-lucas},  ${2^{\ell +1}\choose {2^{\ell}+1}}$ and  ${{2^{\ell +1}+1}\choose {2^{\ell}+1}}$  are even. Then for 
$k\in \{2^{\ell +1},2^{\ell +1}+1\}$, $e({\mathcal H}_{\restriction K})$ and $e({\mathcal H}'_{\restriction K})$ have the same parity for all $k$-element subsets $K$ of $V$.
Hence $w_U$ belongs to the kernel of  $^tW_{2^{\ell}+1\; k}$ over the $2$-element field, for $k\in \{2^{\ell +1},2^{\ell +1}+1\}$. 
By Corollary \ref{cor-lucas} and Theorem \ref{thm Wilson}, 
the rank of 
$W_{2^{\ell}+1\; 2^{\ell +1}}$ modulo $2$ is  $v-1+{v\choose 2^{\ell}+1}-{v\choose 2^{\ell}}$, the rank of $W_{2^{\ell}+1\; 2^{\ell +1}+1}$ modulo $2$ is  ${v\choose 2^{\ell}+1}-v$, the second rank is obvious to obtain; the first one uses this: 
${2^{\ell +1}-i\choose {2^{\ell}+1-i}}$ is odd for $i=1, 2^{\ell}+1$; for $i$ even  ${2^{\ell +1}-i\choose {2^{\ell}+1-i}}$ is even; for $i$ odd, $3\leq i\leq 2^{\ell}-1$, set $j:=i-1$, then $j$ is even and $2\leq j\leq 2^{\ell}-2$, we have 
${2^{\ell +1}-i\choose {2^{\ell}+1-i}}= {{2^{\ell}+2^{\ell}-1-j}\choose {2^{\ell}-j}}$, 
${2^{\ell}-j}=\sum_{q=s}^t \theta_q2^q$ with $\theta_s = \theta_t=1$, thus ${2^{\ell}-1-j}=
\sum_{q=0}^{s-1} 2^q+\sum_{q=s+1}^t \theta_q2^q$, that shows that ${2^{\ell +1}-i\choose {2^{\ell}+1-i}}$ is even. Then   the dimension of $Ker(^tW_{2^{\ell}+1\;  2^{\ell +1}})$ is ${v\choose  2^{\ell}}-v+1$ and  the dimension of  $Ker(^tW_{2^{\ell}+1\; 2^{\ell +1}+1})$ is $v$. Note that 
$(1,1,\ldots ,1)\in Ker(^tW_{2^{\ell}+1\; 2^{\ell +1}})\cap Ker (^tW_{2^{\ell}+1\; 2^{\ell +1}+1})$.\\
Let $X_1,\ldots ,X_r$ be the enumeration of the  $(2^{\ell}+1)$-element subsets of $V$ which appears as rows of the matrices $W_{2^{\ell}+1\; 2^{\ell +1}}$ and  $W_{2^{\ell}+1\; 2^{\ell +1}+1}$. 
For $a\in V$, we set 
$v_a:=(\varepsilon_1,\ldots , \varepsilon_r)$ where $\varepsilon_i=1$ if  $a\in X_i$,  $0$ otherwise.   We have  $^tv_a\in Ker (^tW_{2^{\ell}+1\; 2^{\ell +1}+1})$. 
Note that for all $A\subseteq V$, $A\neq \emptyset$, $\sum_{a\in A} {^tv_a}\neq 0$, indeed if $\sum_{a\in A} {^tv_a}= 0$ then $A\neq V$ and $\vert A\cap X_i\vert$ is even for all $i\in \{1,\ldots ,r\}$. Let $i_0$ be such that $A\cap X_{i_0}\neq \emptyset$. Let $u\in A\cap X_{i_0}$ and $w\in V\setminus A$, set $Y:=(X_{i_0} \setminus \{u\}) \cup \{w\}$. Then $Y$ is some $X_i$, but $\vert A\cap Y\vert$ is odd, that contradicts $\vert A\cap X_i\vert$ even  for all $i$.
So  the family 
$\{^tv_a\ : \ a\in V\}$ is linearly independent and thus forms   a basis of  $Ker(^tW_{2^{\ell}+1\; 2^{\ell +1}+1})$.
Let $u\in Ker (^tW_{2^{\ell}+1\; 2^{\ell +1}})\cap Ker (^tW_{2^{\ell}+1\; 2^{\ell +1}+1})$ with $u\neq 0$. Then $u= \sum_{a\in A} {^tv_a}$ for some non-empty $A\subseteq V$. Since 
$u\in Ker (^tW_{2^{\ell}+1\; 2^{\ell +1}})$,  $\sum_{a\in A} {^tW_{2^{\ell}+1\; 2^{\ell +1}}} {^tv_a}=0$. It follows that  $\sum_{a\in {A\cap F}} \vert \{X_i\ : \ a\in X_i,\ X_i\subseteq F\}\vert =0$ for every  $2^{\ell +1}$-element subset $F$ of $V$ and thus
$\vert A\cap F\vert$ is even. From that we deduce  $A=V$.  Indeed, if $A\neq V$, pick $b\in V\setminus A$, let $F_1$ be a  $2^{\ell +1}$-element subset of $V$ such that $A\cap F_{1}\neq \emptyset$. Let $b_1\in A\cap F_{1}$, set $F:= (F_{1} \setminus \{b_1\}) \cup \{b\}$, then we have $\vert A\cap F\vert$ odd, a contradiction. Thus 
$^tu= \sum_{a\in V} v_a=(1,1,\ldots ,1)$ proving that  $\{^t(1,1,\ldots ,1)\}$ forms a basis of $Ker (^tW_{2^{\ell}+1\; 2^{\ell +1}})\cap Ker (^tW_{2^{\ell}+1\; 2^{\ell +1}+1})$.
Since $^tw_U\in Ker (^tW_{2^{\ell}+1\; 2^{\ell +1}})\cap Ker (^tW_{2^{\ell}+1\; 2^{\ell +1}+1})$, then  
 $w_U=(0,\cdots ,0)$ or  $w_U=(1,\cdots ,1)$, that is  $U$ is empty or complete, so  ${\mathcal H}'={\mathcal H}$ or ${\mathcal H}'=\overline {\mathcal H}$.
 \end{proof}

\begin{problem}
Extend the proof of Theorem \ref{thm h=3, k=4;5} to others values of $h$. For an example, if $h=6$, $s(h)=10$. We guess that $v(h)\leq 16$. Does the same number of edges up to complementation for the $k$-element subsets,  $k=8,9$ and $10$, suffices for the equality of hypergraphs  up to complementation?
\end{problem}

\subsection{Monomorphic decomposition of hypergraphs} We present in this section a notion of monomorphic decomposition of a uniform hypergraph. This notion was introduced in \cite{P-T-2013}. Due to the introduction of an equivalence relation previously considered in \cite{oudrar-pouzet-cras} and developed in  \cite{O}, our presentation is simpler. Let  ${\mathcal H}:=(V,{\mathcal E})$ be a $h$-uniform  hypergraph. Let $F\subseteq V$, we say that ${\mathcal H}$ is $F$-\emph{constant} or that $V\setminus F$ is a \emph{constant block} if $\mathcal H(A)= \mathcal H(A')$  for every  $A, A'\in [V]^h$  such that $A\cap F=A'\cap F$. 
We say that $\mathcal H$ is \emph{almost-constant} if it is $F$-constant for some finite subset $F$. 

\begin{example}
  A $2$-uniform hypergraph is simply a (undirected)  graph.  If $G:=(V, \mathcal E)$ is  a graph then a subset 
$B$ of
$V$ is a constant block if and only if it satisfies the two conditions:
\begin{enumerate}
\item B is either a clique or an independent of $G$;
\item B is an autonomous  subset of $G$, that is for every $y\in {\mathcal E} \setminus B$ and every $x, x'\in B$,
$\{y,x\}\in \mathcal E \Leftrightarrow \{y,x'\}\in \mathcal E$.
\end{enumerate} 
\end{example}

A \emph{monomorphic decomposition} of a $h$-uniform hypergraph $\mathcal H$ is a partition of $V$ into constant blocks. 
\begin{lemma}\label{lem:partition} A partition $\mathcal P$ of $V$ is a monomorphic decomposition if $\mathcal H(A)= \mathcal H(A')$  for every  $A, A'\in [V]^h$ such that $\vert A\cap B\vert =\vert A'\cap B\vert $ for every block $B\in \mathcal P$. 
\end{lemma} 
This is essentially Lemma 2.9 p.13 of \cite{P-T-2013}. 

Let $x,y\in V$. We set  $x\equiv_{\mathcal H} y$ if
\begin{equation}\label {eq:def2}
\mathcal H(K\cup \{x\})=\mathcal H( K\cup \{y\})
\end{equation} 
for every $K\in [V\setminus \{x,y\}]^{h-1}$.

 \begin{proposition} \label{maurice-thiery} The relation $\equiv_{\mathcal H}$ is an equivalence relation on $V$. The blocks of this equivalence are constant. They form a monomorphic decomposition of $\mathcal H$ and every monomorphic decomposition is finer.
 \end{proposition}
 Except for the introduction of the equivalence relation,  this is essentially Lemma 2.11 and Proposition 2.12 p. 14 of \cite{P-T-2013}. We give an outline of the proof below.  

\begin{proof} First, 
$\equiv_{\mathcal H}$ is an equivalence relation. For that it suffices
to check that it is transitive. Let $ x,y, z\in V$ with $x\equiv_{\mathcal H} y$ and $y\equiv_{\mathcal H}z$.
We check that $x\equiv_{\mathcal H} z$. We may suppose these elements pairwise distinct.
Let $K\in [V\setminus \{x,z\}]^{h-1}$. \\
Case 1) $y\not \in K$. In this case
$\mathcal H(K\cup\{x\})= \mathcal H(K\cup\{y\})$  and $\mathcal H(K\cup\{y\})=\mathcal  H(K\cup\{z\})$.
Hence $\mathcal H(K\cup\{x\}= \mathcal  H(K\cup\{z\})$. Thus $x\equiv_{\mathcal H} z$. \\
Case 2) $y\in K$.
Set $K':= (K\setminus \{y\})\cup \{z\}$. Since $x\equiv_{\mathcal H} y$, we have $\mathcal H(K'\cup\{x\})= \mathcal H(K'\cup\{y\})=\mathcal H(K\cup \{z\})$. Similarly, setting $K'':= (K\setminus \{y\})\cup \{x\}$ then, since
$y\equiv_{\mathcal H} z$,  we have $\mathcal H(K\cup\{x\})=\mathcal H(K''\cup\{y\})=\mathcal H(K''\cup \{z\})$. Since
$K'\cup\{x\}=K''\cup\{z\}$, we have $\mathcal H(K\cup\{x\})= \mathcal  H(K\cup\{z\})$ and thus $x\equiv_{\mathcal H} z$ as claimed. Next, the blocks of this equivalence are constant. Let $C$ be a block of $\equiv_{\mathcal H}$. We prove that $\mathcal H(A)= \mathcal H(A')$  for every  $A, A'\in [V]^h$  such that $A\setminus C=A'\setminus C$. Let $\ell := \vert A\setminus A'\vert$. If  $\ell=0$, $A=A'$, there is nothing to prove. If $\ell =1$,  then $A= \{x\}\cup (A\cap A')$ and $A'=\{y\}\cup (A\cap A')$, with $x,y\in C$; in this case $\mathcal H(A)= \mathcal H(A')$ since $x\equiv_{\mathcal H} y$. If $\ell >1$, set $K:=A \setminus C$ and $k:= \vert K\vert$. We may find a sequence of $(h-k)$-element subsets of $C$, say $A_0, \dots A_i, \dots, A_r$ such that $A_0= A\cap C$,  $A_r= A'\cap C$ and the symmetric difference of $A_i$ and $A_{i+1}$ is $0$ or $1$. From the case $\ell=1$ we have $\mathcal H(A_i\cup K)= \mathcal H(A_{i+1}\cup K)$ for $i<r$. Hence $\mathcal H(A)= \mathcal H(A')$. Since these blocks are constant, they form a monomorphic decomposition. To conclude that every other monomorphic decomposition is finer, note that the elements of a constant block are pairwise equivalent w.r.t. $\equiv_{\mathcal H}$ hence contained into a block of this equivalence.  \end{proof}

We call {\it components} the blocks of the  equivalence relation $\equiv_{\mathcal H}$.  Note that the equivalence relations  $\equiv_{\mathcal H}$ and  $\equiv_{\overline{\mathcal H}}$ coincide and also that if $\mathcal H'$ in an other $h$-uniform hypergraph,  every  isomorphism of $\mathcal H$ onto $\mathcal H'$ will transform $\equiv_{\mathcal H}$ into $\equiv_{\mathcal H'}$, thus carrying  the components of  $\mathcal H$ onto the components of $\mathcal H'$.

Let us recall Fra\"{\i}ss\'e's Theorem  on almost constant hypergraphs \cite{Fr2}, this result  is a consequence of the infinite form of Ramsey's Theorem.

\begin{theorem} \label{FRA}
Let $h$ be a non-negative integer, ${\mathcal H}$ be a $h$-uniform hypergraph on an infinite set $V$ and $F$ be a finite subset of $V$. Then there is an infinite subset $V'$ of $V$ such that ${\mathcal H}_{\restriction V'}$ is $F$-constant. 
\end{theorem}

Let $\psi (h):=h+2^t$ where $t$ is the largest integer $t'$ such that $2^{t'}\leq h$, that is $\psi (h)=h+2^{\lfloor \log_2 h\rfloor}$. 
\begin{theorem}\label{seuil}  Let ${\mathcal H}$ and ${\mathcal H}'$ be two $h$-uniform hypergraphs on the same set $V$ of vertices. Suppose that 
\begin{enumerate}[1)] 
\item ${\mathcal H}$ and ${\mathcal H}'$ are $F$-constant for some $F\subseteq V$, 
\item $\vert V\setminus F\vert \geq h$,
\item ${\mathcal H}$ and ${\mathcal H}'$ are $\psi (h)$-hypomorphic up to complementation.
\end{enumerate}
Then   ${\mathcal H}={\mathcal H}'$  or  ${\mathcal H}=\overline{{\mathcal H}'}$. 
\end{theorem}

\begin{proof} 
We may suppose that ${\mathcal H}$ and ${\mathcal H}'$ coincide on $V\setminus F$ (otherwise replace ${\mathcal H}'$ by $\overline{{\mathcal H}'}$) and for example that there is no hyperedge in $V\setminus F$.
We prove that ${\mathcal H}={\mathcal H}'$. For that,  we prove by induction on $\ell$ that:
\begin{equation}\label{equ:induction}
{\mathcal H}(A)={\mathcal H'}(A)
\end{equation}
for every $A\in [V]^h$ such that
$\vert A\cap F\vert \leq \ell$. 

If $\ell:=0$ then since ${\mathcal H}$ and ${\mathcal H}'$ coincide on $V\setminus F$, Equation (\ref{equ:induction}) holds. \\
Suppose that $\ell \geq 1$. Let $A\in [V]^h$ such that  $\vert F\cap A\vert \leq \ell$. 
Let $F_0:=A\cap F$. If $\vert F_0\vert < \ell$ then (\ref{equ:induction}) holds by induction. 
Suppose that $\vert F_0\vert = \ell$. 
Let $S$ and $K:=A\cup S$ where $S\subseteq V\setminus (F\cup A)$, $\vert S\vert =s$, 
and $k:=h+s = \psi (h)$.
Let ${\mathcal H}_0:={\mathcal H}_{\restriction {K}}$, ${\mathcal H}'_0:={\mathcal H}'_{\restriction {K}}$. 
We have $K\in [V]^k$ and $F_0= K\cap F$.\\
Since ${\mathcal H}$ and ${\mathcal H}'$ are $F$-constant, then the $h$-uniform hypergraphs ${\mathcal H}_0$ and  ${\mathcal H}'_0$ are     
$F_0$-constant.\\
The set  ${\mathcal E}({\mathcal H}_0)$ of 
hyperedges of  ${\mathcal H}_0$ is the disjoint union of ${\mathcal E}_{< \ell }({\mathcal H}_0):=\{ A'\in {\mathcal E}({\mathcal H}_0)\ : \ \vert A'\cap F\vert < \ell \}$ and 
${\mathcal E}_{= \ell}({\mathcal H}_0):=\{ A'\in {\mathcal E}({\mathcal H}_0)\ : \ \vert A'\cap F\vert = \ell\}$:
\begin{equation}\label{equ:induction+}
  {\mathcal E}({\mathcal H}_0)= {\mathcal E}_{< \ell}({\mathcal H}_0) \cup {\mathcal E}_{= \ell}({\mathcal H}_0). 
\end{equation}
Similarly, ${\mathcal H}'_0$ decomposes into ${\mathcal E}_{< \ell}({\mathcal H}'_0)$ and 
${\mathcal E}_{= \ell}({\mathcal H}'_0)$:
\begin{equation}\label{equ:induction++}
  {\mathcal E}({\mathcal H}'_0)= {\mathcal E}_{< \ell}({\mathcal H}'_0) \cup {\mathcal E}_{= \ell}({\mathcal H}'_0). 
\end{equation}
By induction hypothesis,
\begin{equation}\label{equ:induction+++}
 {\mathcal E}_{< \ell}({\mathcal H}_0) =   {\mathcal E}_{< \ell}({\mathcal H}'_0).
 \end{equation}
\begin{claim} If ${\mathcal H}_0\simeq {\mathcal H}'_0$ then   ${\mathcal H}_0={\mathcal H}'_0$ and hence (\ref{equ:induction}) holds.
\end{claim}
Indeed, we have   $\vert  {\mathcal E}({\mathcal H}_0)  \vert = \vert  {\mathcal E} ({\mathcal H}'_0)  \vert$.  
 From (\ref{equ:induction+}), (\ref{equ:induction++}) and (\ref{equ:induction+++}), it follows that $\vert  {\mathcal E}_{= \ell}({\mathcal H}_0)  \vert = \vert  {\mathcal E}_{= \ell}({\mathcal H}'_0)  \vert$. 
 Since ${\mathcal H}_0$ is $F_0$-constant, then  
 ${\mathcal E}_{= \ell}({\mathcal H}_0)=\emptyset$ or ${\mathcal E}_{= \ell}({\mathcal H}_0)= \{ F_0\cup I\ : \ I\in [S]^{h-l}\}$. 
 The same holds for ${\mathcal H}'$, hence ${\mathcal E}_{= \ell}({\mathcal H}_0) = {\mathcal E}_{= \ell}({\mathcal H}'_0)$.\\
 Thus ${\mathcal E}({\mathcal H}_0) = {\mathcal E}({\mathcal H}'_0)$ proving 
 ${\mathcal H}_0={\mathcal H}'_0$.\\

To conclude, we prove that ${\mathcal H}_{0}\nsimeq {\overline{{\mathcal H}'_{0}}}$.

{\bf Case 1.}  $k\geq 2l+1$.
Let $C:= K\setminus F_0$. We have $\vert C\vert > \vert F_0\vert$. Since  ${\mathcal H}_0$ is $F_0$-constant, $C$ is a constant block. According to Proposition \ref{maurice-thiery}, $C$ is included into a component $D$ of ${\mathcal H}_0$. Similarly $C$ is included into a component $D'$ of ${\mathcal H}'_0$.
Since $\vert C\vert > \vert F_0\vert$, all components of ${\mathcal H}_0$ distinct of $D$ have a cardinality strictly smaller than $\vert D\vert$. The same with $D'$ the component of $\overline{{\mathcal H}'_0}$ containing $C$.
An isomorphism from   ${\mathcal H}_0$ onto $\overline{{\mathcal H}'_0}$ will map $D$ onto $D'$. But there is no hyperedge of ${\mathcal H}$ in $D$ whereas every $h$-element subset of $D$ is a hyperedge of $\overline{{\mathcal H}'_0}$. Contradiction.

{\bf Case 2.} $k-\ell \leq \ell$.  
By the induction hypothesis  for $\ell'<\ell$ and the fact that  ${\mathcal H}_{0}\simeq {\overline{{\mathcal H}'_{0}}}$, we have:
\begin{equation}\label{eq:parity}
2.\vert  {\mathcal E}_{< \ell}({\mathcal H}_0)  \vert + {k-l\choose h-l}={k\choose h}. 
\end{equation}
Recall that $k=h+s = \psi (h)$ with $s=2^t$ where $t$ is the largest integer $t'$ such that $2^{t'}\leq h$. Then $s\leq h< 2^{t+1}$.
We have 
$h=\sum_{i=0}^{t}a_i2^i$ with $a_i\in \{0,1\}$, $a_t=1$, $k= \sum_{i=0}^{t-1}a_i2^i + 2^{t+1}$.  From Corollary \ref{cor-lucas}, ${k \choose h}$ is even.
Since $\frac{k}{2}\leq \ell  \leq h$ then ${k-\ell \choose h-\ell}$  is odd. Indeed, 
 $\ell \geq s$, thus  $h-l<2^t$, so 
$h-l=\sum_{i=0}^{t-1}c_i2^i$ with $c_i\in \{0,1\}$, and thus $k-l=\sum_{i=0}^{t-1}c_i2^i + 2^t$. Then, by Corollary  \ref{cor-lucas}, ${k-\ell   \choose  h-\ell}$ is odd. 
 The facts that ${k \choose h}$ is even and ${k-\ell \choose h-\ell}$ is odd contradict Equation (\ref{eq:parity}).
\end{proof}

\section{Proof of Theorem \ref{thm1}}
 Suppose that for some $h$ the statement is false. 
\begin{claim} \label{claim:compactness}There are two $h$-uniform hypergraphs $\mathcal H^{(i)}$, $i=1,2$, on $\N^*$  which are not equal to complementation but are  $k$-{hypomorphic up to complementation} for every $k\in \N$. 
\end{claim}

\noindent{Proof of Claim \ref{claim:compactness}.}
We use a compactness argument. 
Let $\mathcal H(h, \N^{*})$ be the set of $h$-uniform hypergraphs with domain $\N^{*}:= \N\setminus \{0\}$. Essentially, this is the power set $\powerset ([\N^{*}]^h)$. Once equipped with the product topology, the power set is a compact space, hence we may view $\mathcal H(h, \N^{*})$ as a compact space and also  the product space $\mathcal H^2(h, \N^{*}):= \mathcal H(h, \N^{*})\times \mathcal H(h, \N^{*})$ of pairs $(\mathcal H^{(1)}, \mathcal H^{(2)})$ of $h$-uniform hypergraph with domain $\N^{*}$. 
Let $\mathcal B$ be the pairs $(\mathcal H^{(1)}, \mathcal H^{(2)})\in  \mathcal H^2(h, \N^{*})$ which are  $k$-{hypomorphic up to complementation} for every $k\in \N$ but whose restrictions  to $\{1, \dots, 2h\}$ are not equal up to complementation. To  prove our claim, it suffices to prove that $\mathcal B$ is non-empty. For that, we prove that this set is the intersection of a decreasing sequence of non-empty closed subsets of $\mathcal H^2(h, \N^{*})$. The compactness of this set ensures the non-emptyness of $\mathcal B$. For every integer $k$, let $\mathcal B_k$ be the set of pairs $(\mathcal H^{(1)}, \mathcal H^{(2)})\in  \mathcal H^2(h, \N^{*})$ such that the restrictions to  $\{1, \dots, 2k-1\}$   are $k$-{hypomorphic up to complementation} and the restrictions  to $\{1, \dots, 2h\}$ are not equal up to complementation.   Trivially,  we have $\mathcal B=\bigcap_{k\in \N}  \mathcal B_{k}$. Due to Theorem \ref{thm:down-k-iso}, we have $\mathcal B_{k+1}\subseteq \mathcal B_{k}$. Since the set of restrictions to $\{1, \dots, 2k-1\}$ of members of $\mathcal H(h, \N^{*})$ is finite, the  $\mathcal B_{k}$'s  are closed. To conclude, it suffices to observe that there are non-empty. 
Let $k\in \N$. Since the statement of the theorem is false, then for  
every  $t$, $k\leq t$, there is some $t'\geq t$ and  two hypergraphs ${\mathcal H}_k^{(i)}$ ($i=1,2$) on $V_{t'}:=\{1,\ldots , t'\}$ which are not equal up to complementation, but are $k$-{hypomorphic up to complementation}.  We can  assume that ${\mathcal H}_k^{(i)}$ ($i=1,2$) are not equal up to complementation on $\{1,\ldots ,l\}$ where $l\leq 2h$ 
(otherwise,  consider the hypergraphs obtained via a permutation transforming a subset $L$ of $V_{t'}$  on which they are not equal up to complementation into $\{1,\ldots ,l\}$). Let $t:= 2k-1$ and $t'$ corresponding to $t$. Let $\mathcal H^{(i)}$ be a $h$-uniform hypergraph on $\N^*$ which extends arbitrarily ${\mathcal H}^{(i)}_{k}$. We have $(\mathcal H^{(1)}, \mathcal H^{(2)})\in \mathcal B_k$. 
\hfill $\Box$

A contradiction is obtained through Theorem \ref{FRA}.  Indeed, there is a subset $F$ with at most $2h$ elements on which $\mathcal H^{1}$ and $\mathcal H^{2}$ are not equal up to complementation. According to  Theorem \ref{FRA}, applied twice, there is an infinite subset $V'$ of $V$ containing $F$ such that  $\mathcal H^{1}_{\restriction V'}$ and $\mathcal H^{2}_{\restriction V'}$ are $F$-constant. These two hypergraphs are $k$-{hypomorphic up to complementation} for all $k$. This contradicts Theorem \ref{seuil}. 

\section{Proof of Theorem \ref{thm-main2}}
Let $\varphi (h):=h+2^{\lfloor \log_2 h\rfloor}-1$. For example $\varphi (h)=2h-1$ if $h$ is a power of $2$. Theorem \ref{seuil} asserts that 
$s(h)\leq \psi(h)=\varphi(h)+1$.  
Theorem \ref{thm2} below asserts that $s(h)\geq \varphi(h)+1$.  Thus $s(h)= h+2^{\lfloor \log_2 h\rfloor}$. This proves $(1)$ of Theorem \ref{thm-main2}. 
Theorem 3.1 and Proposition 4.1 of \cite{dlps1} assert that $s(2)=4$.
We obtain $(2)$ of Theorem \ref{thm-main2} with Theorem \ref{thm h=3, k=4;5}.

 \begin{theorem}\label{thm2} Let $h$ be an integer. There are two $h$-uniform infinite hypergraphs ${\mathcal H}$ and ${\mathcal H}'$ on the same domain such that  ${\mathcal H}\neq {\mathcal H}'$  and  ${\mathcal H}\neq \overline{{\mathcal H}'}$, but ${\mathcal H}$ and ${\mathcal H}'$ are $k$-hypomorphic up to complementation for all $k\leq \varphi (h)$.
 These two hypergraphs are $F$-constant for some subset $F$ of cardinality at most $h$.\end{theorem}

\begin{proof}
Let $V$ be an infinite set, and $F\subseteq V$ having $r$ elements with $2\leq r\leq h$.\\
Let  ${\powerset}^{\star}(F):= {\powerset}(F)\setminus \{\emptyset , F\}$ be the set 
of proper subsets of $F$. Suppose that  $\{{\mathcal A},{\mathcal A'}\}$ is  a partition of ${\powerset}^{\star}(F)$ into two blocks and that $\varphi$ is a permutation of $F$ such that for all $X\in {\powerset}^{\star}(F)$:
\begin{equation}\label{duality}
X\in {\mathcal A} \Longleftrightarrow \varphi(X)\in {\mathcal A'}.
\end{equation}
We define two hypergraphs  ${\mathcal H}$ and ${\mathcal H}'$ as follows:\\
${\mathcal H}:=(V,{\mathcal E})$ where ${\mathcal E}:=\{A\in [V]^h\ :\ A\cap F \in {\mathcal A}\}$\\
and ${\mathcal {H'}} :=(V,{\mathcal E}')$ where ${\mathcal E}':= {\mathcal E} \cup \{A\in [V]^h\ :\ F\subseteq A\}$.

\begin{claim} \label{claim:not equal} ${\mathcal H}\neq {\mathcal H}'$ and ${\mathcal H}\neq {\overline{{\mathcal H}'}}$.
\end{claim}

\begin{proof} Every $h$-element subset of $V$ containing  $F$ belongs to ${\mathcal H}'$, whereas none of these sets is in ${\mathcal H}$. This proves that ${\mathcal H}\neq {\mathcal H}'$. 
Every $h$-element subset of $V\setminus F$ belongs to ${\overline{{\mathcal H}'}}$, none of these sets is in  ${\mathcal H}$ proving 
${\mathcal H}\neq {\overline{{\mathcal H}'}}$.
\end{proof}

\begin{claim}\label{claim:k-iso} ${\mathcal H}$ and ${\mathcal H}'$ are $k$-hypomorphic up to complementation for $k\leq h+r-1$. 
\end{claim}
\noindent{\bf Proof of Claim \ref{claim:k-iso}.}
Since $V$ is infinite, Proposition \ref{thm:down-k-iso} applies. Hence  it suffices to prove this property for $k=h+r-1$. 
Let $K\subseteq V$, with $\vert K\vert = h+r-1$. \\
Case 1. $F\nsubseteq K$. Then clearly ${\mathcal H}_{\restriction K}={{\mathcal H}'}_{\restriction K}$.\\
Case 2.   $F\subseteq K$. Let $\overline{\varphi}$ the map from $K$ into $K$ defined by ${\overline{\varphi}}(x)= {\varphi}(x)$ if $x\in F$, and 
${\overline{\varphi}}(x)= x$ if $x\in K\setminus F$. Clearly, ${\overline{\varphi}}$ is a permutation of $K$.  
We will show that $\overline{\varphi}$ is an isomorphism from 
 ${\mathcal H}_{\restriction K}$ onto  $\overline{{{\mathcal H}'}_{\restriction K}}$. 
 Let $A\in [K]^h$, we will check the equivalence:
 \begin{equation}\label{eq:00}
 A\in {\mathcal E}_{\restriction K}\Longleftrightarrow {\overline{\varphi}}(A)\in \overline{{{\mathcal E}'}_{\restriction K}} \end{equation}
 Note that $A\in {\mathcal E}_{\restriction K}\Longleftrightarrow A\cap F \in {\mathcal A}$. Also, note that 
$A\cap F \in {\mathcal A}\Longleftrightarrow \varphi (A\cap F)\in {\mathcal A}'$.
Since ${\overline{\varphi}}(A) \cap F =    \varphi (A\cap F)$ we have
\begin{equation}\label{eq:1}
 A\in {\mathcal E}_{\restriction K}\Longleftrightarrow {\overline{\varphi}}(A) \cap F \in {\mathcal A}'
 \end{equation}
Now, ${\overline{\varphi}}(A)\in {{{\mathcal E}'}_{\restriction K}}\Longleftrightarrow   {\overline{\varphi}}(A) \cap F \in {\mathcal A} \ \mbox{or}\ F\subseteq A$, that is:
\begin{equation}\label{eq:2}
{\overline{\varphi}}(A)\notin {{{\mathcal E}'}_{\restriction K}}\Longleftrightarrow   {\overline{\varphi}}(A) \cap F \notin {\mathcal A} \ \mbox{and}\ F\nsubseteq A.
\end{equation}

To check (\ref{eq:00}) it suffices to check 
\begin{equation}\label{eq:000}
{\overline{\varphi}}(A) \cap F \in {\mathcal A}' \Longleftrightarrow   {\overline{\varphi}}(A) \cap F \notin {\mathcal A} \ \mbox{and}\ F\nsubseteq A.
\end{equation}

${\overline{\varphi}}(A) \cap F \in {\mathcal A}'\Rightarrow {\overline{\varphi}}(A) \cap F \notin {\mathcal A}$ (because $\{{\mathcal A},{\mathcal A'}\}$ is a partition), and $F\nsubseteq {\overline{\varphi}}(A)$. 
Since $\overline{\varphi}(F)=F$, then  $F\nsubseteq A$. Conversely, since $F\nsubseteq A$, then  
$F\cap A\neq F$. From $\vert F\vert = r\leq h = \vert A\vert$, $F\subseteq K$  and 
$\vert K\vert = h+r-1$, we have $F\cap A\neq \emptyset$. Then  $F\cap A$ is a proper subset of $F$, thus 
${\overline{\varphi}}(A) \cap F$  is a proper subset of $F$.  Since ${\overline{\varphi}}(A) \cap F \notin {\mathcal A}$, then 
${\overline{\varphi}}(A) \cap F \in {\mathcal A}'$.\hfill $\Box$

\begin{claim}\label{claim:power} Let $F$ be a $r$-element subset of $V$, $r\geq 2$. Then there is a partition $\{{\mathcal A},{\mathcal A'}\}$ of ${\powerset}^{\star}(F)$ into two blocks and a permutation  $\varphi$ of $F$ satisfying Equation (\ref{duality})
if and only if $r$ is a power of $2$. 
\end{claim}
\noindent {\bf Proof of Claim \ref{claim:power}.} If there is a partition $\{{\mathcal A},{\mathcal A}'\}$ of ${\powerset}^{\star}(F)$ into two blocks and a permutation  $\varphi$ of $F$ satisfying Equation (\ref{duality}), then for each integer $k$, with $1\leq k\leq r-1$, the number of $k$-element subsets of $F$ is even, thus  ${r \choose k}$ is even for all $k\in \{1, \ldots ,r-1\}$.  Then, by Corollary \ref{cor-lucas}, $r$  is a power of $2$. 

Conversely, if $r$ is a power of $2$, with $r\geq 2$,  we can find a permutation $\varphi$ of a $r$-element set $F$  and a partition of ${\powerset}^{\star}(F)$ into two classes ${\mathcal A},{\mathcal A}'$ such that for each proper subset $K$ of $F$, $\overline{\varphi} (K)$ and $K$ are not in the same class (where $\overline{\varphi} (X) = \{ \varphi (x)\ :\ x \in X \}$).
Take for $\varphi$ a circular permutation of $F$. Fix  an integer $k$. 
Let $\overline \varphi$ be the induced permutation on $[F]^k$.
We have $\overline{\varphi}^r=\overline {\varphi^r}= id_{{\powerset}^{\star}(F)}$. 
Then the order of $\overline {\varphi }$ divides $r$, therefore it is of the form $2^{r'}$. 
It is easy to see that $\overline{\varphi}^r$ has no fixed point. 
Decompose $\overline \varphi$ in cycles, the order of each cycle divides $2^{r}$, so is even.
So each cycle is not trivial with even order. We say that two subsets of $F$ of size $k$ are equivalent if we pass from one to the other by some 
$\overline{\varphi}^s$ with $s$ even. This gives a partition $\{{\mathcal A},{\mathcal A'}\}$ of ${\powerset}^{\star}(F)$ into two blocks satisfying 
Equation (\ref{duality}).
 \hfill $\Box$
\end{proof}

\section{Possible generalizations}
Various kind of isomorphy have been considered for hypergraphs e.g \cite {berge-rado}. Here we consider the following notion. 
Let $W$ be a set and $\mathfrak G$ a subgroup of the group $\mathfrak S(W)$ of permutations of $W$. Let $h$ be an integer. A \emph{$h$-uniform hypergraph valued by $W$} is a pair $\mathcal H:= (V, c)$ where $c$ is a map from $[V]^h$ into $W$. Let  $\mathcal H:= (V, c)$ and  $\mathcal H':= (V, c')$ be  two $h$-uniform hypergraphs valued by $W$; a map $f: V\rightarrow V'$ is an \emph{ isomorphism up to $\mathfrak G$} if there is some  $\sigma\in \mathfrak G$ such that $c'\circ \overline f=\sigma\circ c$ (where $\overline f $ is the extension of $f$ to $[V]^h$).  We say that $\mathcal H$ and  $\mathcal H'$ are  \emph{equal up to $\mathfrak G$} if $V=V'$ and the identity map is an \emph{ isomorphism up to $\mathfrak G$} of $\mathcal H$ onto  $\mathcal H'$ (that is $c'=\sigma \circ c$  for some  $\sigma\in \mathfrak G$).

\subsection{Equality  and isomorphy  up to a permutation group}It is natural to ask if there are two non-negative integers $k$ and $t$, $k\leq t$,  such that two $h$-uniform hypergraphs ${\mathcal H}$ and ${\mathcal H}'$ on the  same set $V$ of vertices, $\vert V\vert \geq t$,  and valued by $W$, are equal up to $\mathfrak G$  whenever ${\mathcal H}$ and ${\mathcal H}'$ are $k$-hypomorphic up to $\mathfrak G$.

The answer is negative in general. Here is an example for $h=2$, $W=\{0,1,2\}$ and $\mathfrak G:=\mathfrak S(W)$ the symmetric group on $W$. 

Take  $V:= \N$, $\mathcal  H:= (V, c)$ where  $c(\{0,1\})=c(\{1,2\})=0, c(\{0,2\})=1, c(\{x, y\})=2$ for all other pairs and $\mathcal  H':= (V, c')$ where  $c'(\{0,1\})=c'(\{0,2\})=0, c'(\{1,2\})=1$ and $c'(\{x, y\})=2$ for all other pairs.

\begin{problem}\label{problem:reconstruction} Is it true that there are two non-negative integers $k$ and $t$, $k\leq t$,  such that two  graphs ${\mathcal G}$ and ${\mathcal G}'$ on the  same set $V$ of vertices, $\vert V\vert \geq t$ and valued by $W$, are isomorphic up to $\mathfrak S(W)$  whenever they  are $k$-hypomorphic up to $\mathfrak S(W)$?
\end{problem}

The integers $k$ and $t$ if they exist will depend upon the cardinality of $W$. We give an example of two labelled graphs showing that $k$ must be at least $\frac{2}{3}\vert W\vert-2$. This example, inspired of \cite{pouzet 79},   somewhat codes  a cylinder and a Moebius band.  

Let $n\in \N$, $n\geq 3$, $W:=(\Z/n\Z\times \Z/3 \Z)\cup\{ 0, 1\}$ and $V:= \Z/n\Z\times \Z/2\Z$. Let $c: [V]^2\rightarrow W$ defined by: $c(\{(i,0), (i,1)\}):=1$,  $c(\{(i,0), (i+1,0)\})=c(\{(i,1), (i+1,1)\}):=(i,0)$, $c(\{(i,0), (i+1,1)\}):=(i, 1)$, $c(\{(i,1), (i+1,0)\}):=(i, 2)$ and  $c(\{u, v\})=0$ for all other pairs. Let $c': [V]^2\rightarrow W$ be obtained from $c$ by changing the values on the pairs $\{(0,i), (1,j)\}$ with $i,j\in \{0,1\}$ so that  $c'(\{(0,0), (1,1)\})= c'(\{(0,1), (1,0)\})=(0, 0)$, $c'(\{(0,1), (1,1)\})=(0, 2)$, and $c'(\{(0,0), (1,0)\})=(0,1)$. 

\begin{lemma} \label{3n+2} The  valued graphs $G:=(V,c)$ and $G':=(V,c')$ are not isomorphic up to 
$\mathfrak S(W)$ 
but  their restrictions to  every proper subset  of $V$ are isomorphic up to 
$\mathfrak S(W)$. 
\end{lemma}

If we set  $k:=2n-1$ then since $\vert W\vert =3n+2$ and $\vert V\vert =2n$, we have $k < \frac{2}{3}\vert W\vert-2$ ($=2n-\frac{2}{3}$), the claim ensures that this values of $k$ is too small to yield an isomorphy up to a permutation group. \\

\noindent{\bf {Proof of Lemma \ref{3n+2}.}} \\
1) $G$ and $G'$ are not isomorphic up to  $\mathfrak S(W)$. 
Indeed, associate to $G= (V, c)$  the  graph $G_c$ whose vertex set is the set $^2V$ of directed pairs $(x,y)$ of distinct elements of $V$ and edges are pairs  $\{e, e'\}$ with $e=(x,y)$, $e'=(x',y')$ such that $c(\{x,y\})=c(\{x',y'\})$, $c(\{x,x'\})=c(\{y,y'\})\not = c(\{x,y\})$,  $c(\{x,y'\})\not=c(\{y,x'\})$, $c(\{x,y'\})$ and $c(\{y,x'\})$ distinct from $c(\{x,y\})$ and $c(\{x,x'\})$. Then, observe that if a directed pair $e:=(x,y)$ is not an isolated vertex in this graph then  $\{x,y\}\in \{ \{(i,0), (i,1)\},  \{(i,0), (i+1,0)\} , \{(i,1), (i+1,1)\}\}$. Next, observe that $G_c$ contains  exactly two  cycles of length $n$; one   made of the vertices $((i,0), (i, 1))$, $ i\in \Z/n\Z$, the other made of vertices $((i,1), (i, 0))$ for $ i\in \Z/n\Z$. Let  $G'_{c'}$ be the graph  defined by the same way. This graph contains a cycle of length $2n$, made of vertices   $((i,j), (i, j+1))$, $(i,j)\in \Z/n\Z\times \Z/2\Z$, the cycle being enumerated as 
$((0,0), (0, 1))$, $((1,1), (1, 0))$,  $((2,1), (2, 0))$,  \ldots , $((n-1,1), (n-1, 0))$, 
$((0,1), (0, 0))$, $((1,0), (1, 1))$,  $((2,0), (2, 1))$,  \ldots ,  $((n-1,0), (n-1, 1))$. There is 
no shorter cycle. Suppose that   $G$ and $G'$ are  isomorphic up to  $\mathfrak S(W)$ via some map $f$ and some permutation of $W$.   The map $f$   induces an isomorphism from  $G_c$ onto  $G'_{c'}$ hence a  $n$-cycle of $G_c$ is sended onto an  $n$-cycle of $G_{c'}$, but there are none. A contradiction. \\
2) For every proper subset $K$ of $V$ the restrictions $G_{\restriction K}$ and $G'_{\restriction K}$ are isomorphic up to $\mathfrak S(W)$.  First, suppose that   $K$ is disjoint from one of the sets $\{(0,0), (0,1)\}$, $\{(1,0), (1,1)\}$  then the two valued graphs are identical hence isomorphic. Next, suppose that $K$ is not disjoint from these sets but does not contain some element of  $L:=\{(0,0), (0,1), (1,0),(1,1)\}$. We claim that in this case, the identity on $K$ and some permutation $\sigma$ of $W$ form an isomorphism up to $\mathfrak S(W)$ 
from $G_{\restriction K}$ onto $G'_{\restriction K}$. Suppose for  an example that $(0,1)\not \in K$ holds, thus $(0,0) \in K$. In this case, set $\sigma$  the permutation of $W$ such that $\sigma((0,0))= (0,1)$, 
$\sigma((0,1))= (0, 0)$ and $\sigma(u)=u$ for every $u\in W\setminus \{(0,0), (0,1)\}$. Finally, suppose that $L\subseteq K$. 
Let $i_1$ be the largest integer such that 
$K_1:= \{1,2,\ldots , i_1 \}\times \Z/2\Z\subseteq K$. 
Since $(1,0), (1,1)\in K$, the integer $i_1$ is well-defined. 
Furthermore,  $i_1<n-1$. Otherwise, since $(0,0)$ and $(0,1)$ belong to $K$, we would have 
$K=\Z/n\Z\times  \Z/2\Z=V$.
From the definition of $i_1$, we have $(i_1+1,0)$ or $(i_1+1,1)\notin K$.
Suppose that $(i_1+1,1)\notin K$. 
The case $(i_1+1,0)\notin K$ will be similar. 
Let $f$ be the map from $K$ to $K$ which is the identity on $K\setminus K_1$ and which exchanges $(i,0)$ and $(i,1)$ on $K_1$ and let $\sigma$ be the permutation of $W$ such that 
$\sigma (i_1,0)=(i_1,2)$, $\sigma (i_1,2)=(i_1,0)$,  $\sigma (i,1)=(i,2)$, $\sigma (i,2)=(i,1)$  
for $i\in \{1,2,\ldots ,i_1-1\}$ and $\sigma (u)=u$ for other elements of $W$. 

Note that 
\begin{equation}\label{eq:001}
\sigma (i,j)\neq (i,j)  \Rightarrow i\in \{1,2,\ldots ,i_1\}.
\end{equation}
We check that $f$ is an isomorphism with respect to $\sigma$, that is:
\begin{equation}\label{eq:0011}
 \sigma (c(u,v))=c'(f(u),f(v))\ \mbox{for every}\ u,v\in K.
 \end{equation}
If $u,v\in K\setminus K_1$, then $f(u)=u$, $f(v)=v$ and due to (\ref{eq:001}), 
$c'(u,v)=\sigma (c(u,v))$ as required by (\ref{eq:0011}). \\
Suppose that $u,v\in K_1$. We have $u=(i,j)$, $v=(i',j')$ and $u':=f(i,j)=(i,j\dot{+}1)$, 
 $v':=f(i',j')=(i',j'\dot{+}1)$ where $\dot{+}$ is the sum modulo $2$.\\
 If $i=i'$, $c'(u',v')=c(u,v)=  \sigma (c(u,v))$ as required by (\ref{eq:0011}). If $i\neq i'$, we may suppose 
 $i'=i\dot{+}1$. There are two cases. First $j=j'$. In this case, $c(u,v)=(i,0)$, $c'(u',v')=(i,0)=
 \sigma (i,0)$ since $i\neq i_1$. Next $j\neq j'$, if $u=(i,1)$ and $v=(i+1,0)$, then $c(u,v)=(i,2)$, $c'(u',v')=(i,1)=  \sigma (i,2)$. The case $u=(i,0)$ and $v=(i+1,1)$ is similar.\\
 Suppose that $u\in K_1$ and $v\in K\setminus K_1$. 
 Hence $u'=(i,j+1)$ and  $v'=v=(i',j')$. We may suppose that either 
($i=1$ and $i'=0$)  
 or  $i=i_1$ and $i'=i_1+1$. 
 In the first case, suppose that $u=(1,1)$. If $v=(1,0)$ then $c(u,v)=(0,0)$, 
 whereas $c'(u',v')=(0,0)= \sigma (c(u,v))$.  
 If $v=(0,0)$ then $c(u,v)=(0,1)$, 
 whereas $c'(u',v')=(0,1)= \sigma (c(u,v))$.  The case  $u=(0,0)$ is similar. 
 In the second case, we have $j'=0$. If $j=1$ then, $c(u,v)=(i_1,2)$ and 
 $c'(u',v')= (i_1,0)=\sigma (c(u,v))$. If $j=0$ then $c(u,v)=(i_1,0)$ and 
  $c'(u',v')=(i_1,2)=\sigma (c(u,v))$. \hfill $\Box$

\begin{problem}
Find examples showing that $k\geq 2\vert W\vert$.
\end{problem}

A special instance of Problem \ref{problem:reconstruction} is this. Let $n$ be a non-negative integer, suppose that $W:=\{0,1\}^n$ and that $\mathfrak G$ is  the permutation group  made of the identity and the involution on $W$ defined by $\overline u:= u+1$ where $1:= (1, \dots, 1)\in W$ and the addition is modulo $2$. A valued graph $G:= (V, c)$ by $W$ identifies to a multigraph; if we define the \emph{complement} of $G$ by setting $\overline G:= (V, \overline c)$ where  $\overline c (\{x,y\}):= \overline {c(\{x,y\})}$, then isomorphy up to complementation means isomorphy up to $\mathfrak G$.  Find $k$ and $t$ such that the conclusion of Problem \ref{problem:reconstruction} holds. For $n=1$ this is the result of \cite{dlps1}. 

\subsection{Isomorphy of hypergraphs up to a permutation group and isomorphy of relational structures}

Associate to each  valued $h$-uniform hypergraph $\mathcal H:= (V, c)$ valued by $W$  the pair $\widehat {\mathcal {H}}:= (V, Eq_c)$ where $Eq_c$ is the kernel of $c$,  that is the equivalence relation defined on $[V]^h$ by $A Eq_c A'$ if $c(A)=c(A')$. Associate also  the  $2h$-ary-structure $\widetilde {\mathcal H}:= (V, \rho_c)$ where $\rho_c$ is the subset of $V^{2h}$ made of $2h$-uples $(x_1, \dots x_{2h})$ such that $c(\{x_1, \dots x_h\})= c(\{x_{h+1}, \dots, x_{2h}\})$. 

We have:

\begin{lemma} \label{isomorphismtypes}Let $\mathfrak G:=\mathfrak S(W)$ be  the symmetric group on $W$. Let $\mathcal  H:= (V, c)$ and $\mathcal H':= (V', c')$ be two $h$-uniform hypergraphs  valued by $W$ and a map $f:V\rightarrow V'$.  
Then the following properties are equivalent:

\begin{enumerate}[(i)]
\item  $f$ is an isomorphism  up to $\mathfrak G$ from $\mathcal H$ onto $\mathcal H'$; 

\item $f$ is bijective and
\begin{equation}
H Eq_c H' \;  \Leftrightarrow \;  f(H) Eq_{c'} f(H')\;  \; \text{for all} \;  H, H'\in [V]^h; 
\end{equation} 

 \item $f$ is an isomorphism from $\widetilde {\mathcal H}$ onto $\widetilde {\mathcal H'}$. 
  \end{enumerate}
\end{lemma}

We  mention that with Lemma \ref{isomorphismtypes} and Theorem \ref{gottlieb-kantor},  Proposition 
 \ref {thm:down-k-iso} is generalized as follows.
 
 \begin{proposition}
Let $W$ be a set and $\mathfrak G$ be a subgroup of the group $\mathfrak S(W)$ of permutations of $W$. Let $v,k$ be non-negative integers, Let  $t \leq min{(k,  v-k)}$ and ${\mathcal H}$ and ${\mathcal H}'$ be two  $h$-uniform hypergraphs, on the same set $V$ of $v$ vertices, valued by $W$.  If ${\mathcal H}$ and ${\mathcal H}'$ are $k$-hypomorphic up to  $\mathfrak G$ then they are $t$-hypomorphic up to $\mathfrak G$.
\end{proposition}

\section{Conclusions}
The motivation of this line of research comes from  several reconstruction results and conjectures about binary structures. The Ulam's reconstruction conjecture, still unsolved, is well-known (see the surveys \cite{BD1,BD2}). Fra\"{\i}ss\'e made a related conjecture about relational structures. The case of binary structures was handled by Lopez.    A \emph{binary structure} can be understood as a pair $R:= (V, c)$ where $c$ is a map from the set 
$V^2$ of ordered pairs $(x,y)$ of elements of $V$ into a set $W$. The notion of isomorphism, and $k$-hypomorphism can be defined as for valued graphs. It was shown by Lopez \cite{lopez} (see also \cite{FrLo}) that if two  binary structures  $R$ and $R'$ on a finite set $V$ are $k$-hypomorphic for all $k\leq 6$  then they are $k$-hypomorphic for every integer $k$. Supposing $W=\{0,1\}$,  Hagendorf and Lopez \cite{hagendorf-lopez} say that two  binary structures $R$ and $R'$ are   \emph{hemimorphic} if either they are isomorphic or one is isomorphic to the dual of the other (the \emph{dual} of $R:= (V, c)$  is $R^d:= (V, c^d)$ where $c^d(x,y):= c(y,x)$). They prove that  hemimorphy  behaves as  hypomorphy,  with a    treshold of $12$ instead of $6$.  Numerous  publications are built  on these results (e.g. \cite {lopez-rauzy, YB, YBG, boussairi}). But for $W:= \{0,1\}^n$ and the corresponding notion of duality, the reconstruction problem is unsolved.  Recently,   Ben Amira, Chaari, Dammak and Si Kaddour \cite{benamira} replace the dual of a binary structure $R:=(V,c)$ (with value in $W= \{0,1\}$)  by its \emph{complement} $R:= (V,  \overline c) $  defined by setting $\overline c(x,y):= 1+ c(x,y)$ (where the addition is defined modulo $2$), they consider the  notion of  isomorphy up to  complementation, and  obtain some encouraging results.  Instead of $W= \{0,1\}$ we may suppose that $W= \{0,1\}^n$,  fix a group of permutations $\mathfrak G$ on $W$ and   ask the same question as in Problem \ref{problem:reconstruction}. As it is easy to see, results mentionned above  are special instances of this question. On an other hand, none of these results   extend  to ternary relations \cite {pouzet 79}. Since the relation associated to a $h$-uniform hypergraph  has arity $2h$, the answer to Problem  \ref{problem:reconstruction} does not seem to follow from general results.


\begin{thebibliography}{1} 
\bibitem{benamira} A. Ben Amira, B. Chaari, J. Dammak  and H. Si Kaddour, A reconstruction problem of digraphs under the hypomorphy up to complementation, work in progress, Lyon, december 2014. 
\bibitem{berge-rado} C. Berge and R. Rado,
Note on isomorphic hypergraphs and some extensions of Whitney's theorem to families of sets. {\em  J. Combinatorial Theory Ser. B} {\bf 13} (1972), 226--241. 
\bibitem {BD1} J. A. Bondy, A graph reconstructor's manual.  \emph{Surveys in combinatorics}, 1991 (Guildford, 1991),  221--252, London Math. Soc. Lecture Note Ser., 166, Cambridge Univ. Press, Cambridge, 1991. 
\bibitem {BD2} J. A. Bondy and R. L. Hemminger, Graph reconstruction--a survey,
\emph{J. Graph Theory} {\bf 1} (1977), no.  3, 227--268.
\bibitem {Bo} J. A. Bondy and U. S. R. Murty,  \emph{Graph theory}. Graduate Texts in Mathematics, 244. Springer, New York, 2008,  xii+651 pp.
\bibitem {YB} Y. Boudabbous, 
{La $5$-reconstructibilit\'{e} et l'ind\'{e}composabilit\'{e} des relations binaires}, 
{\em European J. Combin.} {\bf 23} (2002), no. 5, 507--522.
\bibitem {YBG} Y. Boudabbous and G. Lopez,
{The minimal non-($\leq k$)-reconstructible relations}, 
 \emph{Discrete Math.} {\bf 291} (2005), no. 1-3, 19--40.
\bibitem{boussairi} B. Boushabi and A.  Boussa\"{\i}ri,  Les graphes $(-2)$-monoh\'emimorphes, 
\emph{C. R. Math. Acad. Sci. Paris}  {\bf 350} (2012), no. 15-16, 731--735.
\bibitem {dlps1} J. Dammak, G. Lopez, M. Pouzet and H. Si Kaddour, Hypomorphy of graphs up to complementation, \emph{J. Combin. Theory Ser. B} {\bf 99} (2009), no. 1,  84--96.
\bibitem {dlps2} J. Dammak, G. Lopez, M. Pouzet and H. Si Kaddour,  Boolean sum of graphs and reconstruction up to complementation, \emph{Advances in Pure and Applied Mathematics} {\bf 4} (2013), 315--349. 
\bibitem{Fine} N. J. Fine,   {Binomial coefficients modulo a prime},
\emph{Amer. Math. Monthly} {\bf 54} (1947), no. 1, 589--592. 
\bibitem {Fr2} R. Fra\"{\i}ss\'e, Theory of relations. Revised edition. With an appendix by Norbert Sauer. Studies in Logic and the Foundations of Mathematics, 145. North-Holland Publishing Co., Amsterdam, 2000. ii+451 pp. (1st ed.1986).
\bibitem {FrLo} R. Fra\"{\i}ss\'e and G. Lopez, La reconstruction d'une relation dans l'hypoth\`ese forte: Isomorphie des restrictions \`a chaque partie stricte de la base, SMS, \emph{Les presses de l'universit\'e de Montr\'eal} (1990).
\bibitem{Go}  D. H. Gottlieb, A class of incidence matrices, \emph{Proc. Amer. Math. Soc.} {\bf 17}  (1966), 1233--1237.
\bibitem {hagendorf-lopez} J. G. Hagendorf and G. Lopez, La demi-reconstructibilit\'e des relations binaires d'au moins 13 \'el\'ements, \emph{C.R. Acad. Sci. Paris S\'er. I Math.} {\bf 317} (1993), no. 1, 7--12. 
\bibitem {KA} W. Kantor,  On incidence matrices of finite projection
and affine spaces, \emph{Math.Zeitschrift} {\bf 124} (1972), 315--318.
\bibitem{lopez} G. Lopez, Deux r\'esultats concernant la d\'etermination d'une relation par les types d'isomorphie de ses restrictions, \emph{C.R. Acad. Sci. Paris, S\'erie A}, {\bf 274} (1972), 1525--1528.
\bibitem{lopez-rauzy} G. Lopez and C. Rauzy,
{Reconstruction of binary relations from their restrictions of cardinality 2,3,4 and (n-1),}
II, \emph{Z. fur  Math.  Logik} {\bf 38} (1992), 157--168.
\bibitem{Lucas} E. Lucas, Sur les congruences des nombres eul\'eriens et les coefficients diff\'erentiels des fonctions trigonom\'etriques suivant un module premier, \emph{Bull. Soc. Math. France} {\bf 6} (1878), 49--54.
\bibitem{O} D. Oudrar, Sur l'\'enum\'eration de structures discr\`etes. Une approche par la th\'eorie des relations; document de travail, Universit\'e d'Alger, Ao\^ut 2014.  
\bibitem{oudrar-pouzet-cras} D. Oudrar and M. Pouzet, D\'ecomposition monomorphe des structures relationnelles et profil de classes h\'er\'editaires, 7 pp. 2014,  arXiv:1409.1432  
\bibitem {Pm} M. Pouzet,
Application d'une propri\'et\'e combinatoire des parties d'un ensemble aux
groupes et aux relations, \emph{Math. Zeitschr.}  {\bf 150} (1976), no. 2, 117--134.
\bibitem{pouzet 79} M. Pouzet, Relations non reconstructibles par leurs restrictions,  {\em J. Combin. Theory} Ser. B {\bf 26} (1979), no. 1, 22--34.
\bibitem{P-T-2013} M. Pouzet and N. M. Thi{\'e}ry,
 Some relational structures with polynomial growth and their
  associated algebras {I}: Quasi-polynomiality of the profile,
\emph{Electron. J. Combin.}, {\bf 20}, no. 2,  Paper 1, 35 pp. 
 \bibitem{W} R. M. Wilson,  A diagonal form for the incidence matrices of $t$-subsets vs. $k$-subsets,  {\em European J. Combin.} {\bf 11} (1990), no. 6, 609--615. 

\end{thebibliography}
\end{document}